\documentclass{amsart}
\usepackage{graphicx, amssymb, amsthm, amsfonts, latexsym, epsfig, amsmath}
\usepackage{verbatim}
\usepackage{pstricks,pstricks-add,pst-math}
\usepackage{vaucanson-g}
\usepackage{enumerate}

\newtheorem{thm}{Theorem}[section]
\newtheorem{lem}[thm]{Lemma}

\newtheorem{prop}[thm]{Proposition}
\newtheorem{col}[thm]{Corollary}
\theoremstyle{definition}
\newtheorem{defn}[thm]{Definition}
\newtheorem{eg}[thm]{Example}
\newtheorem{rem}[thm]{Remark}

\DeclareMathOperator{\Int}{int}

\newcommand{\set}[1]{\left\{#1\right\}}

\newcommand{\B}[1]{\ensuremath{\mathbb{#1}}}
\newcommand{\N}{\B{N}}

\newcommand{\ra}{\rightarrow}

\renewcommand{\include}{\input}

\newcommand{\w}{\omega}

\newcommand{\nowt}{\emptyset}

\newcommand{\eps}{\varepsilon}

\newcommand{\Z}{\B Z}

\newcommand{\ol}{\overline}

\newcommand{\iity}{internal chain transitivity}
\newcommand{\pos}{weak incompressibility}

\newcommand{\nat}{\mathbb{N}}
\newcommand{\integ}{\mathbb{Z}}

\newcommand{\andysq}{\begin{flushright}$\square$\end{flushright}}
\newcommand{\cl}[1]{\overline{#1}}
\newcommand{\restr}[2]{#1\hspace{-0.15cm}\restriction_{#2}}
\newcommand{\ontop}[2]{\genfrac{}{}{0pt}{}{#1}{#2}}

\begin{document}

\title{Characterizations of $\omega$-Limit Sets in Topologically Hyperbolic Systems}

\author[A. D. Barwell]{Andrew D. Barwell}
\address[A. D. Barwell]{School of Mathematics, University of Bristol, Howard House, Queens Avenue, Bristol, BS8 1SN, UK -- and -- School of Mathematics, University of Birmingham, Birmingham, B15 2TT, UK}
\email[A. D. Barwell]{A.Barwell@bristol.ac.uk}
\author[C. Good]{Chris Good}
\address[C. Good]{School of Mathematics, University of Birmingham, Birmingham, B15 2TT, UK}
\email[C. Good]{cg@for.mat.bham.ac.uk}
\author[P. Oprocha]{Piotr Oprocha}
\address[P. Oprocha]{AGH University of Science and Technology, Faculty of Applied Mathematics, al. A. Mickiewicza 30, 30-059 Krak\'ow, Poland, -- and -- Institute of Mathematics, Polish Academy of Sciences, ul. \'Sniadeckich 8, 00-956 Warszawa, Poland}
\email{oprocha@agh.edu.pl}
\author[B. E. Raines]{Brian E. Raines}
\address[B. E. Raines]{Department of Mathematics, Baylor University, Waco, TX 76798--7328,USA}
\email{brian\_raines@baylor.edu}


\begin{abstract} 
It is well known that $\w$-limit sets are internally chain transitive and have weak incompressibility; the converse is not generally true, in either case. However, it has been shown that a set is weakly incompressible if and only if it is an abstract $\omega$-limit set, and separately that in shifts of finite type, a set is internally chain transitive if and only if it is a (regular) $\w$-limit set. In this paper we generalise these and other results, proving that the characterization for shifts of finite type holds in a variety of topologically hyperbolic systems (defined in terms of expansive and shadowing properties), and also show that the notions of internal chain transitivity and weak incompressibility coincide in compact metric spaces.
\end{abstract}

\maketitle


\section{Introduction}\label{intro}
Let $X$ be a compact metric space and $f:X\to X$ be a continuous map. The \textit{$\w$-limit set} of a point $x\in X$ is the closed, (strongly) invariant set $\w(x,f)=\bigcap_{k=0}^\infty \overline{\{f^n(x): n\ge k\}}$. Such sets have been studied by many authors, and much is now known about their structure, particularly for maps of the interval. In \cite{Agronsky} it is shown that every closed nowhere dense subset and every finite union of closed subintervals of the unit interval can occur as an
$\w$-limit set for some continuous map. For a given map of the interval, on the other hand, the $\w$-limit sets are characterized in both \cite{Balibrea} and in \cite{Blokh} in terms of topological and dynamical properties. The topological structure of $\w$-limit sets is discussed in \cite{AH,BlockCoppel} and for specific
maps in \cite{BS,goodknightraines,goodrainessua}.

\textit{Internal chain transitivity} and \textit{internal chain recurrence} have applications in the study of economics, epidemiology, game theory, and mathematical biology (see the references and citations of \cite{Hirsch} for numerous examples). In \cite{Hirsch}, Hirsch \emph{et al} study internal chain transitivity in relation to repellors and uniform persistence. They prove that (compact) $\w$-limit sets are internally chain transitive and that the Butler-McGehee Lemma holds for internally chain transitive sets. This allows them to extend various results, which use this lemma, in the study of uniform persistence. The Butler-McGehee Lemma states that whenever an isolated, invariant set $M$ is a proper subset of an $\omega$-limit set $L$ there are points $u,v\subseteq L$ for which $\omega(u),\alpha(v)\in M$. It is
interesting to note that Butler-McGehee type properties have been used in the characterization of $\w$-limit sets appearing in \cite{Agronsky,Balibrea}.

Another property well known to hold in $\w$-limit sets is \textit{weak incompressibility}. This was first observed in \cite{Sarkovskii} by {\v{S}}arkovs'ki{\u\i}, who gave a proof in \cite{Sarkovskii2}. A proof also appears in \cite{BlockCoppel} and weak incompressibility has been mentioned in both \cite{Balibrea} and \cite{Blokh} in connection with characterizations of $\omega$-limit sets. In \cite{Bowen} it is shown that a homeomorphism $f:X\to X$ of a compact metric space is topologically conjugate to the action of a homeomorphism $g:Y\to Y$ on one of its $w$-limit sets ($f$ is an \textit{abstract $\omega$-limit set}) if and only if $X$ has weak incompressibility. In fact, as we show in Section \ref{WI}, {\pos} and {\iity} are equivalent in compact metric spaces.

\medskip

In \cite{Barwell}, symbolic dynamics are used to show that in shifts of finite type the $\w$-limit sets are precisely the internally chain transitive sets.  In \cite{Barwell2}, these symbolic arguments are extended to prove similar results for certain interval maps, in particular that for piecewise linear interval maps with gradient greater than 1, sets which do not contain the image of the critical point are $\w$-limit sets if and only if they are internally chain transitive. In the current paper, we use analytical arguments to isolate relevant properties of shifts of finite type, allowing us to generalise these results in the following theorem, which characterizes $\omega$-limit sets of \textit{topologically hyperbolic} maps (also known as \textit{topologically Anosov} maps), expanding maps, and maps with two types of \textit{pseudo-orbit shadowing}:

\begin{thm}\label{thm:main_ICT_omega}
Assume that $(X,d)$ is a compact metric space, $\Lambda\subseteq X$ is closed and that $f:X\rightarrow X$ is continuous. Assume also that one of the following properties hold:
\begin{enumerate}
	\item $f$ has limit shadowing on $\Lambda$;\label{ICT_omega_cond1}
	\item $f$ is topologically hyperbolic;\label{ICT_omega_cond2}
	\item $f$ is expanding on $\Lambda$ and open on $\Lambda$;\label{ICT_omega_cond4}
	\item $f$ has h-shadowing on $\Lambda$ and is open on a neighbourhood of $\Lambda$.\label{ICT_omega_cond6}
\end{enumerate}
Then for any closed subset $A\subset\Lambda$ the following are equivalent:
\begin{enumerate}
    \item[(a)] $A$ has weak incompressibility;
    \item[(b)] $A$ is internally chain transitive;
    \item[(c)] $A=\omega(x_A,f)$ for some $x_A\in X$.
\end{enumerate}
\end{thm}

Theorem \ref{thm:main_ICT_omega} is proved in Section \ref{sec:limsets} and generalizes results of \cite{Barwell2} and \cite{Barwell}, as is shown in the following corollary:

\begin{col}\label{cor:equiv_tent} Suppose that $f:X\to X$, that $\Lambda$ is a closed subset of $X$ and that either:
\begin{enumerate}
  \item $f$ is a uniformly piecewise linear interval map on $[0,1]$ that takes values 0 or 1 at local extrema;
  \item $f$ is a piecewise linear interval map with gradient modulus greater than 1 and finitely many pieces, and that $\Lambda$ does not contain the image of any critical point;
  \item $f$ is a shift of finite type.
\end{enumerate}
For any closed subset $Y\subseteq\Lambda$ the following are equivalent:
\begin{enumerate}
    \item $Y$ has weak incompressibility;\label{fulltentchar:c1}
    \item $Y$ is internally chain transitive;\label{fulltentchar:c2}
    \item $Y=\omega(x_Y,f)$ for some $x_Y\in X$.\label{fulltentchar:c3}
\end{enumerate}
\end{col}
Note that (2) applies to \textit{tent maps} with gradient between 1 and 2, and that (1) applies to the tent map with gradient equal to 2.

\medskip

Much of the work in this paper relates to shadowing of pseudo-orbits, a notion used by Bowen \cite{Bowen} to study $\omega$-limit sets of a class of diffeomorphism, whilst Coven, Kan and Yorke \cite{Coven} show that shadowing is present in certain maps with specific expansive properties. Similar properties are also at the heart of characterizations of $\omega$-limit sets for interval maps found in \cite{Blokh} and \cite{Balibrea}. Unsurprisingly, it turns out that there is a strong link between pseudo-orbit shadowing and maps for which {\iity} characterizes $\omega$-limit sets; points in an internally chain transitive set $\Lambda$ can be linked by pseudo-orbits contained in $\Lambda$ and certain shadowing properties allow us to find an actual orbit that shadows these pseudo-orbits closely enough for its $\w$-limit set to be $\Lambda$. In Section \ref{shad}, we discuss the various notions of shadowing that we use in Theorem \ref{thm:main_ICT_omega} to make this idea precise. We also consider analytic and topological notions of expansivity that impact on, and in certain cases imply, shadowing properties. In Section \ref{WI}, we introduce some basic definitions and results which relate to $\omega$-limit sets, and in Section \ref{sec:limsets} we prove Theorem \ref{thm:main_ICT_omega}, ending with some examples that illustrate the theory.


\section{Weak Incompressibility, Internal Chain Transitivity and Attractors}\label{WI}

In this paper, unless stated otherwise, $X$ is a compact metric space and $f:X\rightarrow X$ is continuous. Recall that if $\{x_n:n\ge 0\}$ is a sequence of points in $X$, the $\w$-\emph{limit set} of $\{x_n:0\le n\}$ is the set
\[
  \w(\{x_n\}_{n\geq 0})=\bigcap_{k=0}^\infty \overline{\{x_n: n\ge k\}}.
\]
In particular, the $\w$-limit set of a point $x\in X$ is the set $\w(x,f)=\w(\{f^n(x)\}_{n\ge 0})$. We often write $\w(x)$ for $\w(x,f)$ if the context is clear.

The (finite or infinite) sequence $\{x_0,x_1,\ldots\}\subseteq X$ is an  $\eps$-\emph{pseudo-orbit}, for some $\eps>0$, if and only if $d(f(x_n),x_{n+1})<\eps$, for all $n\ge 0$. The (infinite) sequence is an \emph{asymptotic pseudo-orbit} if $d(f(x_n),x_{n+1}) \to 0$ as $n\to \infty$ and is an \emph{asymptotic} $\eps$-\emph{pseudo-orbit} if both conditions hold.

In this section we investigate the following two dynamical properties of $\omega$-limit sets, and show that under certain conditions they describe identical behaviour.

The set $\Lambda\subseteq X$ is \textit{internally chain transitive} (or alternatively $f$ is \textit{internally chain transitive on $\Lambda$}) if for every pair of points $x,y\in \Lambda$ and every $\varepsilon>0$ there is an $\varepsilon$-pseudo-orbit $\{x_0=x,x_1,\ldots,x_m=y\}\subseteq \Lambda$ between $x$ and $y$ {of length $m+1>1$}. In the special case $\Lambda=X$, we say that $f$ (or $X$) is \textit{chain transitive}. If the above only holds whenever $x=y$ (and $m>0$), we say that $\Lambda$ (or $f$) is \emph{internally chain recurrent}. A set $\Lambda\subseteq X$  is \textit{weakly incompressible} (or has \textit{weak incompressibility}) if $M\cap \cl{f(\Lambda\setminus M)}\neq\emptyset$ whenever $M$ is a nonempty, closed, proper subset of $\Lambda$.

Clearly $\Lambda$ is weakly incompressible if and only if $\cl{f(U)}\cap (\Lambda\setminus U)\neq\emptyset$ for any proper, nonempty subset $U\subseteq \Lambda$ which is open in $\Lambda$. {\v{S}}arkovs'ki{\u\i} states in  \cite{Sarkovskii} that $\omega$-limit sets have weak incompressibility (without naming the property). This is probably where this property appears for the first time (see also \cite{Blokh}). Bowen \cite{Bowen} proves that any weakly incompressible set for a homeomorphism is always the $\omega$-limit set of some conjugate map, and a proof that $\w$-limit sets are weakly incompressible for all maps on compact spaces appears in Chapter~VI (Lemma 3) of \cite{BlockCoppel} (the converse of this is false, as we will see below).  We adopt the name weak incompressibility from \cite{Balibrea}.

We say that $\Lambda$ is \textit{invariant} provided $f(\Lambda)=\Lambda$ (some authors use the term \textit{strongly invariant}). The following condition is well known to be true for chain-recurrent sets; the proof in our context is very similar, and is omitted.

\begin{prop}\label{prop:ICTinv}
Let $(X,d)$ be a compact metric space, and $f:X\rightarrow X$ be continuous. If $\Lambda$ is a closed, internally chain transitive subset of $X$, then $\Lambda$ is invariant.
\end{prop}

Hirsch \emph{et al} \cite{Hirsch} show that $\omega$-limit sets are internally chain transitive. In the next result we show that, for compact sets, weak incompressibility is equivalent to internal chain transitivity.

\begin{thm}\label{lem:WI=ICT}
Let $(X,d)$ be a compact metric space, $f:X\rightarrow X$ be continuous and let $\Lambda$ be a closed, nonempty subset of $X$. The following are equivalent:
 \begin{enumerate}
   \item $\Lambda$ is internally chain transitive,\label{lem:WI=ICT:1}
   \item $\Lambda$ is weakly incompressible.\label{lem:WI=ICT:2}
\end{enumerate}
\end{thm}
\begin{proof}
To see that \eqref{lem:WI=ICT:2} implies \eqref{lem:WI=ICT:1}, let $\Lambda$ be weakly incompressible. If $U$ is a proper nonempty open subset of $\Lambda$, let $F(U)=\cl{f(U)}\setminus U$. Since $\Lambda$ is weakly incompressible, $F(U)$ is always nonempty.

Suppose that $x$ and $y$ are in $\Lambda$ and that $\varepsilon>0$. Let ${\mathcal C}$ be a finite cover of $\Lambda$ by $\varepsilon/2$-neighbourhoods of points in $\Lambda$ with no proper subcover, and let ${\mathcal B}=\{C\cap \Lambda\ :\ C\in{\mathcal C}\}$.

If $B_1\in{\mathcal B}$, then unless $B_1=\Lambda$, $F(B_1)\neq\emptyset$, and there is some $B_2\in{\mathcal B}$ such that $B_2\cap\cl{f(B_1)}\neq\emptyset$, hence $B_2\cap f(B_1)\neq \emptyset$. Suppose that we have chosen $B_j\in{\mathcal B}$, $j\leq k$, so that for each $j$ there is some $i\leq j$ such that $B_j\cap f(B_i)\neq\emptyset$. Unless $B_1\cup\ldots\cup B_k=\Lambda$, $F(B_1\cup\ldots\cup B_k)\neq\emptyset$, so there is some $B_{k+1}\in{\mathcal B}$ such that $B_{k+1}\cap f(B_1\cup\ldots\cup B_k)\neq\emptyset$, from which it follows that $B_{k+1}\cap f(B_j)\neq\emptyset$ for some $j<k+1$. Since ${\mathcal B}$ is a minimal finite cover, it follows that for any $B,B'\in {\mathcal B}$ we can construct a sequence $B=B_1,B_2,\ldots,B_n=B'$ such that $B_{j+1}\cap f(B_j)\neq\emptyset$ for each $j<n$.

Now suppose that $x=x_0$, $f(x)\in B$ and $y\in B'$ for some $B,\ B'\in{\mathcal B}$. Then we can construct a sequence $B_1=B,\ldots,B_n=B'$ as above. For $j=1,\ldots, n-1$ choose any $x_j \in B_j\cap f^{-1}(B_{j+1})$, and put $x_n=y$. Then $x_0,\ldots,x_n$ is an $\varepsilon$-pseudo-orbit from $x$ to $y$.

To prove that \eqref{lem:WI=ICT:1} implies \eqref{lem:WI=ICT:2}, assume that $\Lambda$ is internally chain transitive, and suppose that $M$ is a proper, nonempty closed subset of $\Lambda$. Pick $y\in M$ and $x\in \Lambda\setminus M$. For each $n\in\nat$, there is a $1/2^n$-pseudo-orbit from $x$ to $y$. Some $z_n\in \Lambda\setminus M$ is the last point in the pseudo-orbit that is not in $M$, and thus is such that $d(f(z_n),M)<1/2^n$. Since $\Lambda$ is compact, without loss of generality we may assume that $z_n\rightarrow z$ which implies that $f(z)\in M\cap\cl{f(\Lambda\setminus M)}\neq \emptyset$.
\end{proof}

Hirsch \emph{et al} also characterize internal chain transitivity in terms of attractors and in terms of asymptotic pseudo-orbits. The definition they use is that the closed, nonempty invariant set $\Lambda$ is an attractor provided there exists an open set $U\supset\Lambda$ such that $\lim_{n\ra \infty} \sup_{x \in U} d( f^n(x), \Lambda)=0$. This is easily be shown to be equivalent to the following notion, which is closely related to weak incompressibility: the closed, nonempty invariant set $\Lambda$ is said to be an \emph{attractor} if and only if there is an open set $U\supset \Lambda$ such that
\begin{enumerate}
\item $\cl{f(U)}\subseteq U$,
\item $\omega(x,f)\subseteq \Lambda$ for every $x\in U$.
\end{enumerate}
Such sets are also said to be \emph{asymptotically stable in the sense of Lyapunov} \cite{BlockCoppel,Milnor} (we note that there are also a number of other concepts which reflect the idea of attraction).

It is an immediate consequence of Proposition \ref{prop:ICTinv} and Theorem \ref{lem:WI=ICT} that closed, weakly incompressible sets are invariant, so that together with \cite[Lemmas~2.3, 3.1 \& 3.2]{Hirsch}, we have the following.

\begin{col}\label{WI=ICT}
Let $(X,d)$ be a compact metric space, $f:X\rightarrow X$ be continuous
and let $\Lambda$ be a closed, nonempty subset of $X$. The following are
equivalent:
 \begin{enumerate}
   \item $\Lambda$ is internally chain transitive;
   \item $\Lambda$ is weakly incompressible;
   \item $\Lambda$ is invariant and no proper subset of $A$ is an attractor for
   $\restr{f}{\Lambda}$;
   \item $\Lambda$ is the $\w$-limit set of some asymptotic pseudo-orbit
   of $f$ in $X$.
\end{enumerate}
\end{col}


\section{Shadowing and Expansivity}\label{shad}

In light of Corollary \ref{WI=ICT} we would like to have a similar characterization of internally chain transitive sets in terms of $\omega$-limit sets of \emph{real} orbits, as opposed to pseudo-orbits. To this end we discuss in this section various notions of \emph{pseudo-orbit tracing}, or \emph{shadowing}, which allow us to guarantee the existence of a real orbit in the neighbourhood of a pseudo-orbit. Shadowing properties are not easy to check in general, thus we also explore notions of expansivity which imply certain shadowing properties. In what follows, we consider versions of known shadowing and expansivity properties and versions restricted to proper subsets of the space.

Let $\eps>0$, and let $K$ be either $\nat$ or $\{0,1,\ldots,k-1\}$ for some $k\in\nat$. The sequence $\{y_n\}_{n\in K}$ $\eps$-\emph{shadows} the sequence $\{x_n\}_{n\in K}$ if and only if for every $n\in K$, $d(y_n,x_n)<\eps$. Furthermore, we say that the sequence $\{y_n\}_{n\in\nat}$ \emph{asymptotically shadows} the sequence $\{x_n\}_{n\in\nat}$ if and only if $\lim_{n\rightarrow\infty}d(x_n,y_n)=0$. If both conditions hold simultaneously, we say that $\{y_n\}_{n\in\nat}$ \emph{asymptotically $\eps$-shadows} the sequence $\{x_n\}_{n\in\nat}$. If $y_n=f^n(y)$ for every $n\in\nat$ then we say that the point $y$ shadows (in whichever sense is appropriate) the sequence $\{x_n\}_{n\in\N}$.

The standard version of pseudo-orbit tracing (below) appeared in \cite{Bowen}, where it was used by Bowen in the study of $\omega$-limit sets of Axiom A diffeomorphisms.

Let $\Lambda$ be a subset of $X$. We say that $f$ has the \textit{pseudo-orbit tracing property on $\Lambda$} (or \textit{shadowing on $\Lambda$}) if for every $\varepsilon>0$ there is $\delta>0$ such that every infinite $\delta$-pseudo-orbit in $\Lambda$ is $\eps$-shadowed by a point $y\in X$. If this property holds on $\Lambda = X$, we simply say that $f$ has \emph{shadowing}.

\begin{rem}\label{lem:finite_shadowing}
It is easy to see that $f$ has shadowing if and only if for every $\eps>0$ there is a $\delta>0$ such that every finite $\delta$-pseudo-orbit is $\eps$-shadowed.
\end{rem}

Corollary \ref{WI=ICT} refers to \emph{asymptotic} pseudo-orbits, so we also consider a modified version of shadowing relating to such orbits, which comes from \cite{Pil}. For $\Lambda$ a subset of $X$, we say that $f$ has \emph{limit shadowing on $\Lambda$} if for any asymptotic pseudo-orbit $\set{x_n}_{n \in \N} \subseteq \Lambda$ there is a point $y\in X$ which asymptotically shadows $\set{x_n}_{n \in \N}$. If this property holds on $\Lambda = X$, then we say that $f$ has \emph{limit shadowing}.

Since there are many examples of systems possessing the limit shadowing property but not possessing the shadowing property (see \cite{KO2,Pil}), the definition of limit shadowing was extended in \cite{sakai}. We state this version of strong shadowing in a local form, consistent with our previous shadowing definition. We say that \emph{$f$ has s-limit shadowing on $\Lambda \subseteq X$} if for every $\eps > 0$ there is $\delta > 0$ such that the following two conditions hold:
\begin{enumerate}
    \item for every $\delta$-pseudo-orbit $\set{x_n}_{n\in \N}\subseteq \Lambda$ of $f$, there is $y \in X$ such that $y$ $\eps$-shadows $\{x_n\}_{n\in\N}$, and
    \item for every asymptotic $\delta$-pseudo-orbit $\set{z_n}_{n\in \N}\subseteq \Lambda$ of $f$, there is $y \in X$ such that $y$ asymptotically $\eps$-shadows $\{z_n\}_{n\in\N}$.
\end{enumerate}
In the special case $\Lambda=X$ we say that $f$ has \emph{s-limit shadowing}.

The following lemma links limit shadowing to s-limit shadowing; the proof is straightforward and is left to the reader.

\begin{lem}\label{s-limit_and_limit_shad}
Let $(X,d)$ be a compact metric space, and $f:X\rightarrow X$ be continuous. If $\Lambda\subseteq f(\Lambda)\subseteq X$ and $f$ has s-limit shadowing on $\Lambda$ then $f$ also has limit shadowing on $\Lambda$. In particular, if $f$ is surjective and has s-limit shadowing then $f$ also has limit shadowing.
\end{lem}

We introduce another form of shadowing, which we call shadowing with exact hit, or \emph{h-shadowing}. The definition is motivated by the fact that h-shadowing characterizes shifts of finite type in the class of one-sided subshifts (see Remark \ref{rem:posexp_h-shad}). 

\begin{defn}
Let $(X,d)$ be a compact metric space, and $f:X\rightarrow X$ be continuous. We say that $f$ has \textit{h-shadowing on $\Lambda\subseteq X$} if for every $\varepsilon>0$ there is a $\delta>0$ such that for every finite $\delta$-pseudo-orbit $\{x_0,x_1,\ldots,x_m\}\subseteq \Lambda$ there is $y\in X$ such that $d(f^i(y),x_i)<\varepsilon$ for every $i<m$ and $f^m(y)=x_m$. If $\Lambda=X$ then we simply say that $f$ has \emph{h-shadowing}.
\end{defn}

Clearly shadowing (and every variation thereof) is hereditary; if $f$ has shadowing on $\Lambda$ then just by the definition $f$ has shadowing on every set $\Lambda'\subseteq \Lambda$.

In a forthcoming paper we explore to a greater extent the interdependencies between the various notions of shadowing; for now we remark that h-shadowing and shadowing are not equivalent in general, as is shown in Example \ref{eg:h-shadow_and_shadow}. In Example \ref{eg:equiv_tent} we show that the full tent map has h-shadowing; for this we need the following lemma.

\begin{lem}\label{lem:reverse_tracking}
Let $T\colon [0,1]\rightarrow[0,1]$ be the full tent map with slope $2$. There is $\lambda>0$ such that for every $\delta<\lambda$, every integer $n>0$ and every $x,y\in [0,1]$ for which $|T^n(x)-y|<\delta$, there is $z$ such that $|T^i(x)-T^i(z)|<2\delta$ for $i=1,\ldots, n$ and additionally $T^n(z)=y$.
\end{lem}
\begin{proof}
Let $c$ denote the critical point of $T$. We denote by $|J|$ the diameter of an interval $J\subseteq [0,1]$.

First observe that there is $\lambda>0$ such that if $J$ is an interval containing $c$ with $|J|<2\lambda$ then $T^{-1}(J)$ has two connected components, none of them containing $c$.

Fix $\delta<\lambda$, fix any $x,y$ with properties as in the assumptions of the lemma and denote $y_n=y$, $x_i=T^i(x)$. {Notice that for $y_n\neq T(c)$, $|T^{-1}(y_n)|=2$ since $T$ is two-one on such points; let $y_{n-1}\in T^{-1}(y_n)$ be the point closest to $x_{n-1}$ (which is $c$ if $y_n=T(c)$)}. If $y_n=x_n$ then the result follows trivially with $z=x$, so assume that $x_n\neq y_n$.

For simplicity, assume that $x_n<y_n$ (the proof for the second case is identical). Then two possibilities can take place:
\begin{enumerate}
\item if $c\not\in (x_{n-1},y_{n-1})$ then $|x_{n-1}-y_{n-1}|< \delta/2 < \delta$,
\item if $c\in (x_{n-1},y_{n-1})$ then $|x_{n-1}-y_{n-1}| < 2\delta$. Additionally when $n>1$, we denote by $y_{n-2}$ a point in preimage $y_{n-2}\in T^{-1}(y_{n-1})$ closest to $x_{n-2}$ and then by the choice of $\lambda$ we have that $c$ is not in the interval spanned by $x_{n-2},y_{n-2}$. This implies that
\[
    |x_{n-2}-y_{n-2}|\leq \frac{1}{2} |x_{n-1}-y_{n-1}|<\delta.
\]
\end{enumerate}
By induction we construct a sequence $y_0,y_1,\ldots, y_n$ such that $T^i(y_0)=y_{i}$, $|y_i-x_i|<2\delta$ for $i=0,\ldots, n$ and $y_n=y$. It is enough to put $z=y_0$ and the proof is finished.
\end{proof}

\begin{eg}\label{eg:equiv_tent}
The tent map $T$ with slope $2$ has $h$-shadowing. Recall that
\[
T(x)=
\begin{cases}2x & x\in[0,1/2];\\2(1-x) & x\in[1/2,1]\end{cases}
\]
To see that $T$ has $h$-shadowing, note first that it has shadowing \cite{Coven}, and let $\lambda$ be provided by Lemma~\ref{lem:reverse_tracking}. Fix any $0<\eps<\lambda/2$ and use the definition of shadowing to find $\delta$ for $\eps/3$. Decrease $\delta$ if necessary so that $\delta<\eps/3$.

Let $x_0,x_1,\ldots,x_n$ be arbitrary $\delta$-pseudo-orbit and let $y$ be a point which $\eps/3$-traces it. By Lemma~\ref{lem:reverse_tracking} there is a point $z$ such that $|T^i(z)-T^i(y)|<2\delta$ for $i=0,\ldots,n$ and $T^n(z)=x_n$. But then
\[
|T^i(z)-x_i|\leq |T^i(z)-T^i(y)|+|T^i(y)-x_n|<3\delta <\eps
\]
and so the conclusion of Lemma \ref{lem:reverse_tracking} gives us that $T$ has $h$-shadowing.
\andysq
\end{eg}

\begin{eg}\label{eg:h-shadow_and_shadow}
By Remark \ref{lem:finite_shadowing} we immediately see that every map with h-shadowing has shadowing; the converse is not true however. To see this, consider a tent map $T$ with slope less than $2$ and critical point $c=1/2$, with shadowing (many such maps exist -- see \cite{Coven}). Take any pre-image path $\{x_0,x_1,\ldots,x_m=T(c)\}$ ending at the image of the critical point. Let $\eps>0$, then for any $0<\delta<1-T(c)$ let $x_m' = T(c)+\delta/2$, and consider the $\delta$-pseudo-orbit $\{x_0,x_1,\ldots,x_m'\}$; clearly there is no point which $\eps$-shadows this pseudo-orbit with exact hit.
\andysq
\end{eg}

In order to decide whether a map has any form of shadowing, we need to look at various notions of expansion in maps. The idea of an expanding (or expansive map) has been used in many contexts in connection with various dynamical properties of maps, shadowing in particular. In \cite{Coven}, Coven, Kan and Yorke use one notion to prove shadowing in tent maps; in \cite{Urbanski}, Przytycki and Urba{\'n}ski use a different notion to prove shadowing in compact metric spaces. Many maps have expansive properties on a proper subset of the whole space, but not on the space itself, and this local type of expansivity is linked to local shadowing (shadowing on a given subset) and $\omega$-limit sets.

\medskip

For a subset $\Lambda\subseteq X$, we say that $f$ is \emph{open on $\Lambda$} if for every $x\in \Lambda$ and every neighbourhood $U$ of $x$, $f(x)\in\Int(f(U))$. Note that $f$ is open on $\Lambda$ if and only if for every $x\in \Lambda$ there is a neighborhood basis $\set{U_i}_{i\geq 0}$ such that $f(U_i)$ is open, for every $i\geq 0$. This local definition of openness is consistent with the standard definition of an open map, since if $f$ is open on $X$, then $f(U)$ is open for every open set $U$.

The following properties have been studied extensively, and can be found in many texts, including \cite{AH, sakai, Urbanski, PEsakai, Walters}.

We say that $f$ is \emph{positively expansive} (with expansive constant $b>0$) if for any $x,y\in X$ the condition
\[
d(f^n(x),f^n(y))<b \hspace{0.5cm}\mbox{ for every }0\leq n\in\integ
\]
implies that $x=y$.

If $f$ is a surjective map it is said to be \textit{$c$-expansive} (with expansive constant $b'>0$) if for any $x,y\in X$ and any full orbits $\set{x_m}_{m\in\Z}$ and $\set{y_n}_{n\in\Z}$ through $x$ and $y$ respectively the condition
\[
d(x_n,y_n)<b' \hspace{0.5cm}\mbox{ for every }n\in\integ
\]
implies that $x=y$.

Positively expansive maps are clearly $c$-expansive, but the converse is not true in general (an example is the bi-infinite full shift, which is $c$-expansive but not positively expansive).

\begin{thm}\label{thm:LimS}
Let $(X,d)$ be a compact metric space and let $f:X\rightarrow X$ be continuous.
\begin{enumerate}
\item If $f$ is positively expansive then $f$ has shadowing if and only if $f$ has h-shadowing;
\item If $f$ is $c$-expansive then $f$ has shadowing if and only if $f$ has s-limit shadowing.
\end{enumerate}
\end{thm}
\begin{proof}
$(1)$:
If $f$ has h-shadowing then $f$ has shadowing (see Remark \ref{lem:finite_shadowing}). So suppose that $f$ has shadowing, let $\eps<b$ and let $\delta>0$ be provided by shadowing for $\eps$. Fix any $\delta$-pseudo-orbit $\set{x_0,x_1,\ldots, x_m}$ and extend it to the infinite $\delta$-pseudo-orbit
\[
x_0,x_1,\ldots, x_m, f(x_m), f^2(x_m), \ldots
\]
If $z$ is a point which $\eps$-shadows the above pseudo-orbit, then $d(f^{j+m}(z),f^j(x_m))<b$ for all $j\geq 0$ which implies that $f^m(z)=x_m$. Thus $f$ has h-shadowing.

$(2)$:
We have to prove if $f$ has shadowing then it has s-limit shadowing, since the converse implication is trivial. Fix $\eps>0$ and assume that $\eps<b/2$ where $b$ is the expansive constant. Let $\delta>0$ be a constant provided by the shadowing property for $\eps$. Shadowing implies that the first part of the definition of s-limit shadowing holds. To prove the second part, let $\set{x_n}_{n\in \N}$ be an asymptotic $\delta$-pseudo-orbit that is $\eps$-shadowed by the point $z$.

Suppose, for a contradiction, that $d(f^n(z),x_n)$ does not converge to $0$ as $n\to\infty$. Since $X$ is compact (so that every sequence has a convergent subsequence), there are points $p_0$ and $q_0$ in $X$ and an infinite subset $N_0$ of $\N$ such that
\begin{enumerate}
  \item $\lim_{n\to\infty, n\in N_0}d\big(f^{n}(z),x_{n}\big)=\eta>0$,
  \item $\lim_{n\to\infty, n\in N_0}f^{n}(z)= p_0$, and
  \item $\lim_{n\to\infty, n\in N_0}x_{n}= q_0$.
\end{enumerate}
By continuity,
$$\lim_{\ontop{n\to\infty}{n\in N_0}}f^{n+k}(z)= p_k=f^k(p_0)$$
for all $k\ge0$. Since
$$
d\big(x_{n+1},f(q_0)\big)
\le d\big(x_{n+1},f(x_{n})\big)+d\big(f(x_{n}),f(q_0)\big),
$$
continuity and the fact that $\{x_n\}$ is an asymptotic pseudo-orbit imply that
$$\lim_{\ontop{n\to\infty}{n\in N_0}}x_{n+1}=q_1=f(q_0).$$
Hence $\lim_{n\to\infty, n\in N_0}x_{n+k}= q_k=f^k(q_0)$ for all $k\ge0$.

Since $X$ is compact, there are points $p_{-1}$ and $q_{-1}$ and an infinite subset $N_{-1}$ of $N_0$ such that
$$
\lim_{\ontop{n\to\infty}{n\in N_{-1}}}f^{n-1}(z)=p_{-1}
\qquad\text{and}\qquad
\lim_{\ontop{n\to\infty}{n\in N_{-1}}}x_{n-1}=q_{-1}.
$$
Again, continuity and the fact that $\{x_n\}$ is an asymptotic pseudo-orbit imply that $f(p_{-1})=p_0$ and $f(q_{-1})=q_0$. Repeating this argument we can find points $p_{-1},p_{-2},\dots$, $q_{-1},q_{-2},\dots$, and infinite sets $N_{-1}\supseteq N_{-2}\supseteq\dots$, such that for all $0<k\in\N$
\begin{enumerate}
\item $0\le n-k$ for all $n\in N_{-k}$,
  \item $\lim_{n\to\infty, n\in N_{-k}}f^{n-k}(z)=p_{-k}$ and
  $f(p_{-k})=p_{-k+1}$,
\item $\lim_{n\to\infty, n\in N_{-k}}f^{n-k}(z)=p_{-k}$ and $f(q_{-k})=q_{-k+1}$.
\end{enumerate}

Now $\{p_k\}_{k\in\Z}$ and $\{q_k\}_{k\in\Z}$ are full orbits passing through $p_0$ and $q_0$ respectively. Moreover
$$
d(p_k,q_k)\le
\begin{cases}
  \sup_{n\in N_0}d\big(f^{n+k}(z),x_{n+k}\big),&\text{if }k\ge0,\\
\sup_{n\in N_k}d\big(f^{n+k}(z),x_{n+k}\big),&\text{if }k<0.\\
\end{cases}
$$
Since $\eps<b/2$ and $z$ $\eps$-shadows $\{x_n\}$, $d(p_k,q_k)<b/2$ for all $k\in\Z$. It follows by $c$-expansivity that
$$
0=d(p_0,q_0)=\lim_{\ontop{n\to\infty}{n\in N_0}}d(f^{n}(z),x_{n})=\eta>0,
$$
which is the required contradiction.
\end{proof}

Notice that if $f$ is a positively expansive surjection then the properties of shadowing, h-shadowing and s-limit shadowing are equivalent. Part $(2)$ of the above result is a natural generalization of the results of \cite{sakai,PEsakai}.

\begin{rem}\label{rem:posexp_h-shad}
The assumptions of Proposition~\ref{thm:LimS} (1) are fulfilled by every open, positively expansive map \cite{PEsakai}; an example is a positively expansive homeomorphism, however this case is trivial since the space must be finite \cite{PEHomeo}. A nontrivial class of positively expansive open maps is the class of one-sided shifts of finite type \cite{KurkaBook}. Walters \cite{Walters} showed that shift spaces have shadowing if and only if they are of finite type. Since maps with h-shadowing have shadowing, and one-sided shifts of finite type are positively expansive, we see that one-sided shifts of finite type are characterized by h-shadowing in the class of shift spaces.
\end{rem}

\begin{defn}
$f$ is said to be \textit{topologically hyperbolic} if it is both $c$-expansive and has shadowing.
\end{defn}

There is a large class of topologically hyperbolic maps. The classical example is an Axiom~A diffeomorphism restricted to its non-wandering set (see \cite{Bowen} for example). Other important classes are shifts of finite type (one or two-sided), and topologically Anosov maps (see \cite{Yang}). A list of conditions equivalent to topological hyperbolicity in the context of homeomorphisms can be found in \cite{sakai} (see also \cite{Ombach,Ombach2}). As we see in Theorem~\ref{thm:main_ICT_omega}, $\omega$-limit sets are fully characterized in the context of topologically hyperbolic maps by \iity.

We note that to obtain such a characterization of $\omega$-limit sets in terms of topological hyperbolicity, the assumption of shadowing can't be dropped on its on, since there are $c$-expansive maps without shadowing for which internal chain transitivity does not characterize $\omega$-limit sets. One such class of maps are chain-transitive sofic shifts which are not transitive (see \cite{kazda} for detailed description of this class), and thus cannot be the $\omega$-limit set of any of the inner points \cite{BlockCoppel}.

\medskip

The standard definition of an expanding map is the following (see also \cite{sakai, Urbanski, PEsakai}), which is generally a stronger property than either positively expansive or $c$-expansive, and will enable us to demonstrate the existence of shadowing properties in various maps. This property can be observed in many classes of maps, such as interval maps away from their turning points.

For a closed set $\Lambda$, we say that $f$ is \textit{expanding on $\Lambda$} if there are $\delta>0$, $\mu>1$ such that $d(f(x),f(y))\geq \mu d(x,y)$ provided that $x,y\in \Lambda\subseteq X$ and $d(x,y)<\delta$. In the case that $\Lambda=X$ we simply say that $f$ is expanding. If there is some open set $U\supset\Lambda$ such that the definition of expanding holds for every $x,y\in U$, we say that $f$ is \textit{expanding on $U$}, or if the set $U$ is not specified, we say $f$ is \textit{expanding on a neighbourhood of $\Lambda$}.

\begin{rem}\label{rem:exp_equiv}
If $f$ is expanding on $\Lambda$ then for each $x\in \Lambda$ there is an open set $U\ni x$ such that $\restr{f}{U\cap \Lambda}$ is one-to-one. Furthermore, if $f$ is expanding on an invariant set $\Lambda$ then it is easy to see that $f$ is positively expansive on $\Lambda$, and also $c$-expansive.
\end{rem}

Przytyicki and Urba{\'n}ski \cite{Urbanski} define a property they refer to as \textit{expanding at $\Lambda$}, which is equivalent to our notion of expanding on a neighbourhood of $\Lambda$; the following is from their text (Corollary 3.2.4):

\begin{lem}\label{lem:Urban_lemma}
Let $(X,d)$ be a compact metric space, and $f:X\rightarrow X$ be continuous. If $f$ is open and expanding then $f$ has shadowing.
\end{lem}

We get the following easy corollary from Proposition \ref{thm:LimS} (1), Remark \ref{rem:exp_equiv} and Lemma \ref{lem:Urban_lemma}:

\begin{col}\label{col:uniexp_open_hshad}
Let $(X,d)$ be a compact metric space, and $f:X\rightarrow X$ be continuous. If $f$ is open and expanding then
\begin{enumerate}
    \item $f$ is topologically hyperbolic;
    \item $f$ has h-shadowing.
\end{enumerate}
\end{col}


\section{Proof of Theorem \ref{thm:main_ICT_omega}}\label{sec:limsets}

In this section we prove our main theorem. To complete the theory we require a property introduced in \cite{Balibrea} (Definition \ref{def:dyn_ind}), which seems closely linked to shadowing but better approximates the dynamics of maps on their $\omega$-limit sets (recall that a set $\Lambda\subset X$ is said to be \textit{regularly closed} if $\Lambda=\ol{\Int \Lambda}$).

\begin{defn}\label{def:dyn_ind}
For a compact metric space $X$ and a continuous map $f\colon X\rightarrow X$ we say that a set $\Lambda\subseteq X$ is \textit{dynamically indecomposable} if for every $\varepsilon>0$, every pair of points $x,y\in\Lambda$ and every pair of open sets $U,V$
such that $x\in U$ and $y\in V$ there is $m>0$ and a sequence of regularly closed sets $J_0,J_1,\ldots,J_m$ for which
\begin{enumerate}
    \item $x\in \Int J_0,\ J_0\subseteq U$,
    \item $J_{i+1}\subseteq f(J_i)$ for $i=0,1,\ldots,m-1$,
    \item $J_i\subseteq B_{\varepsilon}(\Lambda)$ for $i=0,1,\ldots,m$,
    \item $y\in \Int(J_m),\ J_m\subseteq V$.
\end{enumerate}
\end{defn}

Next we present Lemmas \ref{CINEP2} and \ref{thm:idecomp}, which relate dynamical indecomposability to shadowing and $\omega$-limit sets.

\begin{lem}\label{CINEP2}
Let $f\colon X\ra X$ be a continuous map acting on a compact metric space $(X,d)$. If $\Lambda\subseteq X$ is internally chain transitive, $f$ has h-shadowing on $\Lambda$ and is open on a neighbourhood of $\Lambda$, then $\Lambda$ is dynamically indecomposable.
\end{lem}
\begin{proof}
Let $\varepsilon>0$, pick $x,y\in\Lambda$ and let $U$ and $V$ be open with $x\in U$ and $y\in V$. Certainly there is an $\eta>0$ for which
$B_{\eta}(x)\subseteq U$ and $B_{\eta}(y)\subseteq V$. There is also $\xi$ such that $f$ is open on $B_\xi(\Lambda)$. Denote
$\varepsilon'=\min\set{\eta,\varepsilon,\xi/2}$. Let $\delta$ be provided for $\eps'/2$ by h-shadowing. By the assumptions $\Lambda$ is internally chain transitive, so there is a $\delta$-pseudo-orbit $\{x_0=x,x_1,\ldots,x_m=y\}\subset\Lambda$. Thus there is a $z\in X$ for
which $d(f^i(z),x_i)<\varepsilon'/2$ and $f^m(z)=x_m=y$.

So let $J_0=\cl{B_{\varepsilon'/2}(x_0)}$ and for $i=0,1,\ldots,m-1$, let
\[J_{i+1}=\cl{f(J_i)\cap B_{\varepsilon'/2}(x_{i+1})}\]

We claim that
\begin{itemize}
    \item $x\in \Int( J_0),\ J_0\subseteq U$;
    \item $J_{i+1}\subseteq f(J_i)$ for $i=0,1,\ldots,m-1$;
    \item $J_i\subseteq B_{\varepsilon}(\Lambda)$ for $i=0,1,\ldots,m$;
    \item $y\in \Int(J_m),\ J_m\subseteq V$.
\end{itemize}
Notice first that $J_i$ is regularly closed for every $i=0,1,\ldots,m$, since $f$ is open on each $J_i$ and $J_0$ is regularly closed. Moreover, the first condition holds by the definition of $J_0$, the second and third by the definition of the $J_i$.

Since $f$ is open at $f^i(z)\in B_\xi(\Lambda)$, $f^i(z)\in \Int(J_i)\neq\emptyset$, and there is $0<r<\eps'/2$ such that for $i=0,1,\ldots,m-1$
$$
f(\Int (J_i))\supset B_r(f^{i+1}(z))
$$
and in particular,
\[y=f^m(z)\in f(\Int (J_{m-1}))\cap B_r(x_m)\subset\Int(J_m).\]
This proves that the claim holds, and as an immediate consequence we see that $\Lambda$ is dynamically indecomposable.
\end{proof}

\begin{lem}\label{thm:idecomp}
Assume that $(X,d)$ is compact, $f\colon X \rightarrow X$ is continuous, and $\Lambda\subseteq X$ is a closed set which is dynamically indecomposable for $f$. Then $\Lambda=\omega(x_{\Lambda},f)$ for some $x_{\Lambda}\in X$.
\end{lem}
\begin{proof}
Since $\Lambda$ is compact, there is a sequence of points $\{z_n\ :\ n\in\nat\}$ in $\Lambda$ such that $\Lambda=\cl{\{z_n\}_{n\in\nat}}$. Enumerate the collection $\{B_{1/p}(z_n)\ :\ n,p\in\nat\}$ as $\{B_k\ :\ k\in\nat\}$, then for every $k$ there is an $n_k\in\nat$ such that $z_{n_k}\in B_k$. We define a sequence of natural numbers $\{m_n\ :\ n\in\nat\}$ and a sequence of regularly closed sets $\{J_k\ :\ k\in\nat\}$ as follows.
\begin{enumerate}
    \item Let $m_1=1$, let $J_{m_1}$ be the closure of any basic open subset of $B_1$ such that $z_{n_1}\in \Int J_{m_1}$
    \item Given $J_{m_i}$ such that $z_{n_i}\in \Int J_{m_i}$ consider the point $z_{n_{i+1}}\in B_{i+1}$. Since $\Lambda$ is dynamically indecomposable, we can define basic open sets $I_{m_i}^0$ and $\{I_j\ :\ m_i+1\leq j\leq m_{i+1}\}$ whose closures $J_{m_i}^0$ and $\{J_j\ :\ m_i+1\leq j\leq m_{i+1}\}$ respectively are contained in $B_{1/i}(\Lambda)$, and for which $z_{n_i}\in \Int J_{m_i}^0\subseteq \Int J_{m_i}$, $J_{m_i+1}\subseteq f(J_{m_i}^0)$, $J_{j+1}\subseteq f(J_j)$ for $j=m_i+1,\ldots,m_{i+1}-1$, and $z_{n_{i+1}}\in \Int J_{m_{i+1}}\subseteq B_{i+1}$.
\end{enumerate}
By the construction of the $J_k$'s, for every $k\in\nat$ there is a closed set $D\subseteq J_{k-1}$ such that $f(D)=J_k$. Hence, for every $k\in\nat$ there is a $J^{(k)}\subseteq J_0$ such that $f^k(J^{(k)})=J_{k}$. The $J^{(k)}$ are nested, so by compactness $K=\bigcap_{k\in\nat}J^{(k)}\neq\emptyset$.

For $x_{\Lambda}\in K$, $f^i(x_{\Lambda})\in J_i$ for every $i\in\nat$, so certainly $\Lambda\subset\omega(x_{\Lambda},f)$. Suppose that $z\in X\setminus\Lambda$, then there are disjoint open sets $U$ and $V$ for which $z\in U$ and $\Lambda\subseteq V$. Since $\bigcup\{J_j\ :\ m_i\leq j\leq m_{i+1}\}\subseteq B_{1/i}(\Lambda)$ there is an $N\in\nat$ for which $f^n(x_{\Lambda})\in V$ for every $n\geq N$, hence $z\notin\omega(x_{\Lambda},f)$. Thus $\Lambda=\omega(x_{\Lambda},f)$.
\end{proof}

\begin{rem}
Dynamical indecomposability is not a sufficient condition for shadowing of any type. Indeed an irrational rotation of the circle has neither shadowing nor limit shadowing, but it is easy to verify that it is dynamically indecomposable.
\end{rem}

We are now in a position to prove the main result in our paper, Theorem \ref{thm:main_ICT_omega}, which gives various cases in which $\omega$-limit sets are characterized by internal chain transitivity, and thus also weak incompressibility.

\begin{proof}[\textbf{Proof of Theorem \ref{thm:main_ICT_omega}}]
In every case, we get that the closed sets with weak incompressibility are precisely the closed sets with internal chain transitivity by Lemma \ref{WI=ICT}, and furthermore every $\omega$-limit set is internally chain transitive as was shown in \cite{Hirsch}. Thus to prove the theorem we show that a closed internally chain transitive set is necessarily an $\omega$-limit set in each case.

In case (\ref{ICT_omega_cond1}), we have that the closed, internally chain transitive set $Y$ is the $\omega$-limit set of an asymptotic pseudo-orbit $\{x_n\}_{n\in\nat}$ by Corollary~\ref{WI=ICT}, and by limit shadowing there is a point $x_{Y}\in X$ whose orbit asymptotically shadows $\{x_n\}_{n\in\nat}$. Thus $Y=\omega(x_{Y},f)$.

In case (\ref{ICT_omega_cond2}), notice that $f$ is $c$-expansive, so Lemma \ref{s-limit_and_limit_shad} and Theorem \ref{thm:LimS} $(2)$ imply that $f$ has limit shadowing on $Y$, so the proof follows as in case (\ref{ICT_omega_cond1}).

In case (\ref{ICT_omega_cond4}), notice that by Proposition \ref{prop:ICTinv} we have that $Y$ is invariant. Since $f$ is expanding and open on $Y$, by Corollary \ref{col:uniexp_open_hshad} we get that $f$ has h-shadowing on $Y$, so by Lemma \ref{CINEP2} $f$ is dynamically indecomposable, and by Lemma \ref{thm:idecomp} there is some $x_Y\in X$ such that $Y=\omega(x_{Y},f)$.

Case (\ref{ICT_omega_cond6}) follows directly from Lemmas \ref{CINEP2} and \ref{thm:idecomp}.
\end{proof}

\textbf{Corollary \ref{cor:equiv_tent}} applies Theorem \ref{thm:main_ICT_omega} to specific types of maps. Shifts of finite type are well-studied and definitions can be found in many texts, including \cite{Barwell, kazda}. We say that an interval map $f:[0,1]\rightarrow[0,1]$ is \textit{piecewise linear} if there is a set of points $\{c_0=0,c_1,\ldots,c_m=1\}$ such that $f$ is linear on $[c_{i-1},c_i]$ for $0<i\leq m$; a piecewise linear map $f$ is \textit{uniformly piecewise linear} if the gradient modulus of $f$ is everywhere greater than 1, and equal on each of the $m$ subintervals \cite{Chen}; an example of such a map is a tent map with slope $\lambda\in(1,2]$.

To see that Corollary \ref{cor:equiv_tent} holds, note first that all of these maps are open on the set $\Lambda$ as given. \ref{cor:equiv_tent} part (1) follows from a result in \cite{Chen} which shows that uniformly piecewise linear maps that take values 0 or 1 at local extrema have shadowing, and thus have h-shadowing as in Example \ref{eg:equiv_tent}. The result now follows from Theorem \ref{thm:main_ICT_omega} (\ref{ICT_omega_cond6}). \ref{cor:equiv_tent} parts (2) and (3) (which are equivalent symbolically) follow from Theorem \ref{thm:main_ICT_omega} (\ref{ICT_omega_cond4}) since the map is expanding on $\Lambda$ in each case.

\medskip

We end with two examples. The first shows that the characterization of $\w$-limit sets by internal chain transitivity is not hereditary.

\begin{eg}\label{eg:sofic_ICT}
Consider the sofic shift $X$ consisting of all bi-infinite words in $a$, $b$, $c$ and $d$ obtained by following paths in the following presentation.
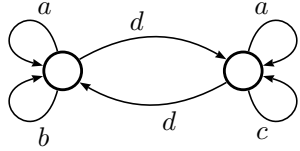
\begin{center}
\VCDraw{%
\begin{VCPicture}{(0,-4)(8,0)}
\State{(2,-2)}{A}
\State{(6,-2)}{B}
\LoopNW{A}{a}
\LoopSW{A}{b}
\LoopSE{B}{c}
\LoopNE{B}{a}
\LArcL{A}{B}{d}
\LArcL{B}{A}{d}
\end{VCPicture}}
\end{center}
The set
$$
\Lambda=\set{a,b}^\Z \cup \set{a,c}^\Z
$$
is closed and shift invariant. Both sets $\set{a,b}^\Z$, $\set{a,c}^\Z$ are internally chain transitive with nonempty intersection, so $\Lambda$ is also internally chain transitive. But $\Lambda$ is not $\omega$-limit set of any point under $\restr{\sigma}{X}$, since any point $x\in X$ such that $\omega(x,\restr{\sigma}{X})\supset \Lambda$ must contain infinitely many symbols $d$ and so there must be a point in $\omega(x,\restr{\sigma}{X})$ which contain $d$ on at least one position. On the other hand $\Lambda$ is the $\w$-limit set of a point in the full shift on $\{a,b,c,d\}$.
\andysq
\end{eg}

The second example shows that there is no general characterization of $\w$-limits sets in terms of internal chain transitivity together with even very strong mixing properties.

\begin{eg} \label{exact_map}
Consider the function $f\colon [-2,2]\to [-2,2]$, whose graph is the piecewise linear curve passing through the points $(-2,2)$, $(-3/2,-2)$, $(-1,0)$, $(-1/2,-2)$, $(1/2,2)$, $(1,0)$, $(3/2,2)$ and $(2,-2)$ (see Figure \ref{exact_map_graph}).  Note that the absolute value of the gradient (at non-critical points) of the function is at least 4. Note also that the function $f$ is topologically exact (locally eventually onto), because if $U$ is any open interval then clearly for some $n>0$, $f^n(U)$ will contain two consecutive critical points, from which it follows that $f^{n+2}(U)=[-2,2]$.

Let $H_{-}=\{0\}\cup\{-1/4^{-n}: n\geq0\}$ and $H_{+}=\{0\}\cup\{1/4^{-n}: n\geq0\}$. Clearly $H_-$ and $H_+$ are both closed, invariant and internally chain transitive sets, because $f(\pm1)=0$ and $f(\pm1/4^{n+1})=\pm1/4^n$. Since $H_-\cap H_+\neq\nowt$, the union $H=H_-\cup H_+$ is also, therefore, closed, invariant and internally chain transitive. However, $H$ is not the $\w$-limit set of any point. To see this we argue as follows. Suppose that $H=\omega(x,f)$ for some $x\in[-2,2]$, and notice that $f\big([0,7/4]\big)=[0,2]$, whilst $f\big([7/4,2]\big)=[-2,0]$. Then since the orbit of $x$ must approach both $H_+$ and $H_-$, for infinitely many $n\in\nat$ we have $f^n(x)\in(0,2]$ and $f^{n+1}(x)\in[-2,0)$, thus we must have infinitely many $n\in\nat$ for which $f^n(x)\in(7/4,2]$, which is disjoint from $H$. But this would mean there is a point $x\in \omega(x,f)\cap[7/4,2]$, and thus $H\neq\omega(x,f)$.
\andysq
\end{eg}

\begin{figure}[ht]
\newrgbcolor{cccccc}{0.9 0.9 0.9}
\psset{xunit=1.5cm,yunit=1.5cm,algebraic=true,dotstyle=*,dotsize=3pt 0,linewidth=0.8pt,arrowsize=3pt 2,arrowinset=0.25}
\begin{pspicture*}(-3.3,-2.5)(3.0,3.0)
\psline[linecolor=cccccc, linewidth=0.9pt](-2,2)(-2,-2)
\psline[linecolor=cccccc, linewidth=0.9pt](2,-2)(2,2)
\psline[linecolor=cccccc, linewidth=0.9pt](2,2)(-2,2)
\psline[linecolor=cccccc, linewidth=0.9pt](-2,-2)(2,-2)
\psline[linecolor=cccccc, linewidth=0.9pt](-2,1)(2,1)
\psline[linecolor=cccccc, linewidth=0.9pt](-2,-1)(2,-1)
\psline[linecolor=cccccc, linewidth=0.9pt](-1,-2)(-1,2)
\psline[linecolor=cccccc, linewidth=0.9pt](1,-2)(1,2)
\psline[linecolor=cccccc, linewidth=0.9pt](-2,-2)(2,2)
\psaxes[xAxis=true,yAxis=true,Dx=1,Dy=1,ticksize=-2pt 0,subticks=2]{->}(0,0)(-2.5,-2.5)(2.5,2.5)
\psline[linewidth=1.1pt](-0.5,-2)(0.5,2)
\psline[linewidth=1.1pt](0.5,2)(1,0)
\psline[linewidth=1.1pt](1,0)(1.5,2)
\psline[linewidth=1.1pt](1.5,2)(2,-2)
\psline[linewidth=1.1pt](-0.5,-2)(-1,0)
\psline[linewidth=1.1pt](-1,0)(-1.5,-2)
\psline[linewidth=1.1pt](-1.5,-2)(-2,2)
\psline[linestyle=dotted, dotsep=1pt](0.06,0.25)(0.06,0)
\psline[linestyle=dotted, dotsep=1pt](0.06,0.06)(0.02,0.06)
\psline[linestyle=dotted, dotsep=1pt](-0.06,-0.25)(-0.06,0)
\psline[linestyle=dotted, dotsep=1pt](0.25,1)(0.25,0)
\psline[linestyle=dotted, dotsep=1pt](0.25,1)(1,1)
\psline[linestyle=dotted, dotsep=1pt](1,1)(1,0)
\psline[linestyle=dotted, dotsep=1pt](-1,-1)(-1,0)
\psline[linestyle=dotted, dotsep=1pt](-0.25,0)(-0.25,-1)
\psline[linestyle=dotted, dotsep=1pt](-0.25,-1)(-1,-1)
\psline[linestyle=dotted, dotsep=1pt](-0.25,-0.25)(-0.06,-0.25)
\psline[linestyle=dotted, dotsep=1pt](0.25,0.25)(0.06,0.25)
\psline[linestyle=dotted, dotsep=1pt](-0.06,-0.06)(-0.02,-0.06)
\end{pspicture*}
\caption{The graph of the function $f$ from Example \ref{exact_map}}\label{exact_map_graph}
\end{figure}
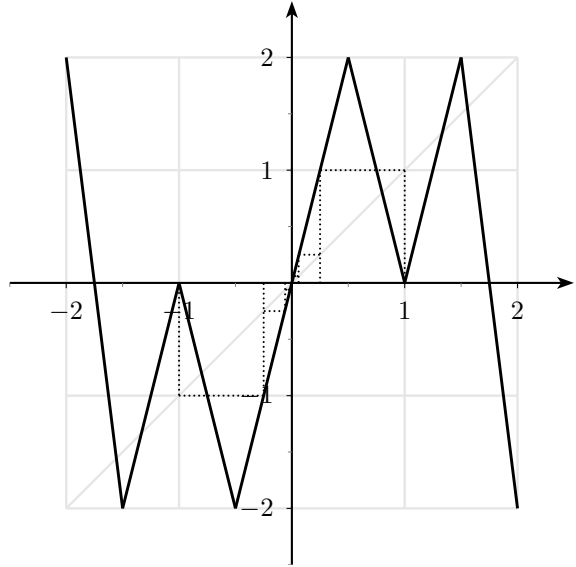


\section*{Acknowledgements}

The authors gratefully acknowledge the useful comments made by Henk
Bruin during early drafts of this paper.

Good received a London Mathematical Society
Collaborative Small Grant to fund Oprocha's visit to Birmingham to
work on results in this paper.
The research of Oprocha leading to results included in this paper was supported by the Marie Curie European Reintegration Grant of the European Commission under grant agreement no. PERG08-GA-2010-272297. He was also supported by the Polish Ministry of Science and Higher Education (2011).
Raines was supported by NSF grant 0604958.
The financial support of these institutions is hereby gratefully acknowledged.

The graphs of interval maps were drawn and exported to PSTricks using the free dynamics mathematics software GeoGebra. The representation of the sofic shift was generated using VauCanSon-G \LaTeX\ package.

\bibliographystyle{plain}
\bibliography{BibtexW-limitsets}

\end{document}